\documentclass[12pt,letterpaper]{amsart}

\usepackage{accents}
\usepackage{enumitem}
\usepackage{graphicx}
\usepackage{caption}
\usepackage{amsmath}
\usepackage{float}

\newtheorem{theorem}{Theorem}[section]
\newtheorem{lemma}[theorem]{Lemma}
\theoremstyle{definition}
\newtheorem{definition}[theorem]{Definition}

\newtheorem{corollary}[theorem]{Corollary}

\theoremstyle{remark}

\newtheorem{claim}[theorem]{{\bf Claim}}

\numberwithin{equation}{section}

\usepackage{color}

\begin{document}

\def\C{\mathbb C}
\def\R{\mathbb R}
\def\X{\mathbb X}
\def\Z{\mathbb Z}
\def\Y{\mathbb Y}
\def\Z{\mathbb Z}
\def\N{\mathbb N}
\def\cal{\mathcal}
\def\cD{\cal D}
\def\tD{\tilde{{\cal D}}}
\def\F{\cal F}
\def\tf{\tilde{f}}
\def\tg{\tilde{g}}
\def\tu{\tilde{u}}

\def\cal{\mathcal}
\def\b{\mathcal B}
\def\c{\mathcal C}
\def\cc{\mathbb C}
\def\x{\mathbb X}
\def\r{\mathbb R}
\def\T{\mathbb T}
\def\uu{(U(t,s))_{t\ge s}}
\def\vv{(V(t,s))_{t\ge s}}
\def\xx{(X(t,s))_{t\ge s}}
\def\yy{(Y(t,s))_{t\ge s}}
\def\zz{(Z(t,s))_{t\ge s}}
\def\ss{(S(t))_{t\ge 0}}
\def\tt{(T(t,s))_{t\ge s}}
\def\rr{(R(t))_{t\ge 0}}
\title[Diffusive Delay in Nicholson Blowflies Models]{\textbf{EXISTENCE OF TRAVELING WAVES in a Nicholson Blowflies Model WITH DELAYED DIFFUSION TERM   }}

\author[Barker]{William Barker}
\address{Department of Mathematics and Statistics, University of Arkansas at Little Rock, 2801 S University Ave, Little Rock, AR 72204. USA}
\email{wkbarker@ualr.edu}

\thanks{}

\date{\today}
\subjclass[2000]{Primary: 35C07 ; Secondary: 35K57 }
\keywords{Traveling waves; Reaction-diffusion equations; Delay; Nicholson Blowflies Equation}
\begin{abstract} In this paper we consider traveling waves for a diffusive Nicholson Blowflies Equation with different discrete time delays in the diffusion term and birth function. We construct quasi upper and lower solutions via the monotone iteration method. This also allows for the construction of $C^2$ upper and lower solutions, and then traveling wave solutions. We then provide numerical results for the  kernel for the iteration.   
\end{abstract}

\maketitle
\section{Introduction}
\noindent The diffusive Nicholson Blowflies Model with a single delay in the birth term is of the form 

\begin{equation}\label{NDE}
\frac{\partial u(t,x)}{\partial t}=\frac{\partial^2 u(t,x)}{\partial x^2}-\delta u(t,x)+pu(t-\tau_1,x)e^{-au(t-\tau_1,x)}
\end{equation}
for $x\in \R, t\ge 0.$  $u(t, x)$ is the blowfly population in space at a certain time, $t.$ The death rate is $\delta>0,$ impact rate on the immature population is $p>0$ due to birth,  and $\tau_1>0$ is the maturation  delay. For more information about the model see Nicholson's groundbreaking work, \cite{Nick1, Nick2} and the  adaptation to include spatial diffusion, \cite{SoYang1, SoYang2, SoYang3}. The above model has been studied in several celebrated papers.

\medskip
\noindent
In 2001, So and Zou, \cite{SoZou} provided an elegant result providing a construction of upper and lower solutions when there is delay in the birth function. This extended the seminal results on the existence of traveling wave solutions via the so called monotone iteration method put forth by Wu and Zou, \cite{wuzou}. In the standard monotone or quasi monotone iteration methods, it is often important to construct upper and lower wave front solutions. This idea was extended by Ma, \cite{ma} who developed super and sub solutions, which relaxed the requirements in Wu and Zou.

\medskip
\noindent
The dynamics of traveling waves for the Nicholson Blowflies model with distributed delay was discussed in 2000 by Gourley and Ruan, \cite{GouRuan} by using energy methods and by a comparison principle for functional differential equations. In a recent paper, the global attractivity of the positive steady state was studied for a non-monotone model with distributed delay by Deng and Wu, \cite{DengWu}.  

\medskip
\noindent 
In 2004, Mei \textit{et al.}, \cite{MeiSo} discussed the  nonlinear stability of traveling wavefronts of a time-delayed
diffusive Nicholson blowflies equation under a weighted $L^2$ norm. This result is very interesting and was extended by Lin and Mei, \cite{LinMei} who also provided numerical results via a finite difference method. Furthermore, recent results by Huang and Liu, \cite{HuangLiu} and Huang and Xu \cite{HuangXu} studied the existence of traveling waves with a birth function given two different delays. 

\medskip
\noindent
Boumenir and Nguyen, \cite{boumin} developed the idea of quasi upper and lower solutions via a modified Perron Theorem, see Theorem (4.1) in Mallet-Peret, \cite{mal}. More information can also be found in Pruss, \cite{pru}.  This leads us to the motivation of this paper. The idea of placing the diffusion term is a relatively new idea. In fact, it was shown in Barker and Nguyen, \cite{BarkNguy} traveling waves exist for reaction diffusion equations with discrete delay in both the diffusion and reaction term of the form \begin{equation}\label{RD1}
\frac{\partial u(x,t)}{\partial t}=D\frac{\partial^2 u(x,t-\tau_1)}{\partial x^2} +f(u_t),
\end{equation}
where $t\in\R, \tau_1>0,  x, u(x,t)\in \R, \ D>0, \ f:C\left([-\tau_2,0], \R\right)\to \R$ is continuous and $u_t(x)\in C\left([-\tau_2,0], \R\right),$ defined as 
\[u_t(x)=u(x, t+\theta), \ \theta\in [-\tau_2,0], \ t\ge 0, \ x\in \R.\]
Where $f$ is Lipschitz continuous and 
\[f(0)=f(K)=0, \ \text{and} \ f(u)\neq 0, \ 0<u<K.\]
under certain monotone conditions.

\medskip
\noindent
This paper will be organized as follows: Section 2 will lay out some preliminary notation that may be used in the later sections. Our main results will be stated in Section 3. The existence of traveling waves will be discussed as well as solutions to exponential type polynomials of second order. The last part of Section 3 will be the explicit construction of quasi upper and lower solutions for the Nicholson Blowflies Equation with diffusive delay. Section 4 will cover some numerical results for the convolution kernel and the resulting upper and lower solutions.  

\section{Preliminaries}
In this paper we will use some standard notations such as $\R,\C$ standing for the fields of reals and complex numbers. $\Re z$ and $\Im z$ denote the real part and imaginary part of a complex number $z$. The space of all bounded and continuous functions from $\R  \to \R^n$ is denoted by $BC(\R,\R^n)$ which is equipped with the sup-norm $\| f\| := \sup_{t\in\R} \| f(t)\|$. $BC^k(\R,\R^n)$ stands for the space of all $k$-time continuously differentiable functions $\R\to\R^n$ such that all derivatives up to order $k$ are bounded. 

\medskip
\noindent
If the boundedness is dropped from the above function spaces we will simply denote them by $C(\R,\R^n)$ and $C^k(\R,\R^n)$. We will use the natural order in $BC(\R,\R^n)$ that is defined as follows: 
For $f,g\in BC(\R,\R^n)$ we say that $f\le g$ if and only if $f(t)\le g(t)$ for all $t\in \R$, and we will say that $f< g$ if $f(t)\le g(t)$ for all $t\in \R$, and $f(t)\not= g(t)$ for all $t\in \R$. 
\section{Main Results}
\noindent We consider the following diffusive delayed Nicholson Blowflies model
\begin{equation}\label{NDE1}
\frac{\partial u(t,x)}{\partial t}=\frac{\partial^2 u(t-\tau_1,x)}{\partial x^2}-\delta u(t,x)+pu(t-\tau_2,x)e^{-au(t-\tau_2,x)}
\end{equation}

\medskip
\noindent
Moreover, we assume that $1<p/\delta\le e,$ then we can find two equilibria
\[U_0=0, \ U_{e}=\frac{1}{a}(\ln{(p/\delta)}).\]
We are interested in the question:  Is there a traveling wave front connecting the two equilibrium of Eq. (\ref{NDE1})? 

\medskip
\noindent
To this end, we assume the wave translation form as $\xi=x+ct$ where $c$ is a positive wave speed. Applying the transformation $\phi(\xi)=u(t,x)$ to  Eq. (\ref{NDE1}) gives 
\begin{equation}
\phi^{\prime\prime}(\xi-c\tau_1)-c\phi^{\prime}(\xi)-\delta \phi(\xi)+p \phi(\xi-c\tau_2)e^{-a \phi(\xi-c\tau_2)}=0.
\end{equation}
Moving the delay out of the diffusion term via the transformation $\zeta=\xi-r_1,$ where $r_1=c\tau_1, r_2=c\tau_2$ yields the equivalent model
\begin{equation}
\phi^{\prime\prime}(\zeta)-c\phi^{\prime}(\zeta+r_1)-\delta \phi(\zeta+r_1)+p \phi(\zeta+(r_1-r_2))e^{-a \phi(\zeta+(r_1-r_2))}=0.
\end{equation}
For simplicity let $\zeta=t$ and the model becomes
\begin{equation}\label{WaveA}
\phi^{\prime\prime}(t)-c\phi^{\prime}(t+r_1)-\delta \phi(t+r_1)+p \phi(t+(r_1-r_2))e^{-a \phi(t+(r_1-r_2))}=0.
\end{equation}

\medskip
\noindent
We invoke the following monotone conditions on $f$.  \begin{enumerate}
\item[(H1)] $f(\hat 0) =f(\hat K)=0,$ where $\hat 0$, ($\hat U_e$, respectively) is the constant function $\phi (\theta) =0$ ($\phi (\theta )=U_e$, respectively), for all $\theta \in [-\tau_2,0]$;
\item[(H2)] There exists a positive constant $\beta$ such that
$$
f(\varphi)-f(\psi ) +\beta (\varphi(0)-\psi (0)) \ge 0
$$
for all $\varphi, \psi \in C([-\tau_2,0],\R)$ with $0\le \varphi (s)\le \phi (s)\le U_e$ for all $s\in [-\tau_2,0]$;
\end{enumerate}
Using the monotone conditions (H1) and (H2) we can write Equation (\ref{WaveA}) as
\begin{equation}\label{Wave1}
    \phi^{\prime\prime}(t)-c\phi^{\prime}(t+r_1)-\delta \phi(t+r_1)-\beta\phi(t+r_1)+H(\phi(t))=0,
\end{equation}
\noindent where $H(\phi(t))=p \phi(t+(r_1-r_2))e^{-a \phi(t+(r_1-r_2))}+\beta\phi(t+r_1).$
\medskip
\noindent
The asymptotic behavior is \[\lim_{t\to -\infty}\phi(t)=U_0, \ \lim_{t\to \infty}\phi(t)=U_{e}.\]

\begin{definition} A function $\varphi \in BC^2(\R,\R)$ is called an upper solution (lower solution, respectively) for the wave equation (\ref{Wave1}) if it satisfies the following
\begin{align*}
& D\varphi''(t)-c\varphi'(t+r_1)+f^c(\varphi_{t+r_1})\le 0 , \\
& (D\varphi''(t)-c\varphi'(t+r_1)+f^c(\varphi_{t+r_1})\ge 0, \ \mbox{respectively})
\end{align*}
for all $t\in\R$.
\end{definition}

\medskip
\noindent
Here,
$f_{c}\in \X_c:= :C([-c\tau_2,0],\mathbb{R}^{n})\rightarrow\mathbb{R}$,
defined as
\[
f_{c}(\psi)=f(\psi^{c}),\quad\psi^{c}(\theta):=\psi(c\theta),\quad\theta
\in\lbrack-\tau_2,0].
\]
\begin{definition} A function $\varphi \in C^1(\R,\R),$ where $ \varphi, \varphi'$ are bounded on $\R$, $\varphi''$ is locally integrable and essentially bounded on $\R$ (that is, $\varphi''\in L^\infty$), is called a quasi- upper solution (quasi-lower solution, respectively) for the wave equation (\ref{Wave1}) if it satisfies the following for almost every $t\in \R$
\begin{align*}
& D\varphi''(t)-c\varphi'(t+r_1)+f^c(\varphi_{t+r_1})\le 0 , \\
& (D\varphi''(t)-c\varphi'(t+r_1)+f^c(\varphi_{t+r_1})\ge 0, \ \mbox{respectively}) .
\end{align*}
\end{definition}

 \medskip
 \noindent
The quadratic equation $\lambda^2-c\lambda-\delta=0,$ which has two distinct real roots \[\lambda_1=\frac{c-\sqrt{c^2+4\delta}}{2}<0, \ \lambda_2=\frac{c+\sqrt{c^2+4\delta}}{2}>0,\]
whenever $c>2\sqrt{\delta}.$ This allows for the following lemma. 
\begin{lemma} Let $c>2\sqrt{\delta}$ and consider the characteristic equation for Eq. (\ref{Wave1}) 
\begin{equation}\label{CE}
 \lambda^2-c\lambda e^{r_1\lambda}-\delta e^{r_1\lambda}=0.   
\end{equation}
Define an open strip $U\subset \C$ by $\{z\in \C : \lambda_1-\varepsilon\le\Re{z}\le 0\}.$ Then For sufficiently small $r_1, \varepsilon>0$ 
\begin{enumerate}
\item the characteristic equation has no roots on the imaginary axis,

\item the characteristic equation has a single root continuously dependent on $r_1$ in $U,$ denoted as $\eta_1(r_1).$ 
\end{enumerate} 
\end{lemma}
\noindent This is a special case of Proposition (3.1) from {\cite{BarkNguy}}.
This allows us to see that we have a unique bounded solution from the modified Perron Theorem, see {\cite{mal}}. Thus, we have the following results from {\cite{BarkNguy}}.
The solution for Eq. \ref{Wave1} is given by \begin{equation}\label{Solution}
 F\left(\phi\right)=\int^\infty_{-\infty} G(t-s,r)H(\phi(s))ds,   
\end{equation}
where there exists some positive constants $M_1,\delta_1$ such that for all $t\in \R,$ $|G(t,r)| \le M_1e^{-\delta_1 |t|}\\$ 
  
\begin{lemma}
Let $(\phi)$ be a quasi- upper solution (quasi-lower solution, respectively) of Eq. (\ref{Wave1}). Then, $F(\phi)$ is an upper solution (lower solution, respectively) of Eq. (\ref{Wave1}).
\end{lemma}
\begin{theorem}
Assume $(H1)$ and $ (H2)$  hold, if there is an upper $\overline{\phi}$ and a lower solution $\underline{\phi}$ for Eq.(\ref{Wave1}) in $\Gamma$  such that for all $t\in \R$
\[0\le  \underline{\phi}(t)\le \overline{\phi}(t).\]
Then, there exists a monotone traveling wave solution to the system (\ref{Wave1}).
\end{theorem}
\subsection{Quasi Upper Solutions} We will now explicitly find an acceptable quasi upper solution. The quadratic equation $\mu^2-c\mu+p=0,$ which has two  real roots $\mu_1<0<\mu_2$
whenever $c>2\sqrt{p}.$ In fact, we need the following lemma.
\begin{lemma}
Let $c>2\sqrt{p}$ and consider the equation 
\begin{equation}\label{CE1}
 \mu^2-c\mu e^{r_1\mu}+ pe^{r_1\mu}=0,   
\end{equation}
and define an open strip $V\subset \C$ by $\{z\in \C : 0<\Re{z}<\mu_2+\varepsilon\}.$ Then for sufficiently small $r_1, \varepsilon>0$ such that $V\cap\{\mu_2\}=\emptyset.$ Then Eq. (\ref{CE1}) has a single root continuously dependent on $r_1$ in $V,$ denoted as $\eta_2(r_1).$  Moreover,  $\eta_2(r_1)$ is real and
\begin{equation}
\lim_{r_1\to 0} \eta_2(r_1) =\mu_2.
\end{equation}
\end{lemma}

\begin{claim}\label{claim 6}
For sufficiently small $r_1$ and $c>2\sqrt{p}$, and $\eta_2$ to be the root of Eq. (\ref{CE1}) in $V$, then  the function \[\varphi_1(t)=\left\{
\begin{array}
[c]{l}%
\frac{U_e}{2}e^{\eta_{2}t},\;\;\;\;\;\;\quad t\leq t_0,\\
U_{e}(1-\frac{1}{2}e^{-\eta_{2}t}),\quad t>t_0
\end{array}
\right. \]  is a quasi-upper solution of (\ref{Wave1}).
\end{claim}
\begin{proof}
From elementary calculus it is easy to see
\begin{align*}
\varphi_{1}^{\prime}(t) &  =\left\{
\begin{array}
[c]{l}%
\frac{U_e\eta_2}{2}e^{\eta_2t},\quad t\leq 0,\\
\frac{U_e\eta_2}{2}e^{-\eta_2 t},\quad t> 0
\end{array}
\right.  ,\quad\varphi_{1}^{\prime\prime}(t)=\left\{
\begin{array}
[c]{l}%
\frac{U_e\eta_2^{2}}{2}e^{\eta_2t},\quad t\leq 0,\\
\frac{-U_e\eta_{2}^{2}}{2}e^{-\eta_2t},\quad t> 0
\end{array}
\right. 
\end{align*}
Note that $\varphi_1'$ is continuous and bounded on $\R$ and $\varphi_1''$ exist and is continuous everywhere and bounded except for $t=0$. We will take $U_e=1,$ for brevity, since it appears in every term  The proof can now be completed in three cases.\\
{\bf Case 1 $t< - r_1 $}: We have 
\begin{align*}
&\phi^{\prime\prime}(t)-c\phi^{\prime}(t+r_1)-\delta \phi(t+r_1)+p \phi(t+(r_1-r_2))e^{-a \phi(t+(r_1-r_2))}\\
&=\frac{\eta_2^{2}}{2}e^{\eta_2t}-\frac{c\eta_2}{2}e^{\eta_2(t+r_1)}-\frac{\delta}{2} e^{\eta_2(t+r_1)}+\frac{p}{2}e^{\eta_2(t+r_1)}e^{-a \phi(t+(r_1-r_2))}\\
&\le\left(\frac{\eta_2^{2}}{2}-\frac{c\eta_2}{2}e^{\eta_2r_1}+\frac{p}{2}e^{\eta_2r_1}\right)e^{\eta_2t}-\frac{\delta}{2}  e^{\eta_2(t+r_1)}\\
&=-\frac{\delta}{2}  e^{\eta_2(t+r_1)}\le 0.
\end{align*}
{\bf Case 2 $- r_1\le t\le 0 $}: We have 
\begin{align*}
&\phi^{\prime\prime}(t)-c\phi^{\prime}(t+r_1)-\delta \phi(t+r_1)+p \phi(t+(r_1-r_2))e^{-a \phi(t+(r_1-r_2))}\\
&=\frac{\eta_{2}^{2}}{2}e^{\eta_2t}-\frac{c\eta_2}{2}e^{-\eta_2 (t+r_1)}-\delta \phi(t+r_1)+\frac{p}{2}e^{\eta_{2}(t+r_1-r_2)}e^{-a \phi(t+(r_1-r_2))}\\
&\le \frac{\eta_{2}^{2}}{2}e^{\eta_2t}-\frac{c\eta_2}{2}e^{-\eta_2 (t+r_1)}-\delta \phi(t+r_1)+\frac{p}{2}e^{\eta_{2}(t+r_1)}\\
&=\frac{\eta_{2}^{2}}{2}e^{\eta_2t}-\frac{c\eta_2}{2}e^{\eta_2 (t+r_1)}+\frac{p}{2}e^{\eta_{2}(t+r_1)} +\frac{c\eta_2}{2}e^{\eta_2 (t+r_1)}-\frac{c\eta_2}{2}e^{-\eta_2 (t+r_1)}-\delta \phi(t+r_1)\\
&\left(\frac{\eta_2^{2}}{2}-\frac{c\eta_2}{2}e^{\eta_2r_1}+\frac{p}{2}e^{\eta_2r_1}\right)e^{\eta_2t}+c\eta_2\sinh(r_1\eta_2)-\delta \phi(t+r_1)\\
&=c\eta_2\sinh(r_1\eta_2)-\delta \phi(t+r_1),
\end{align*}
because $\eta_2$ is a root of Eq. (\ref{CE1}).
When $r_1\to 0, \ c\eta_2\sinh(r_1\eta_2)\to 0.$ Thus, there is a small enough $r_1$ such that \[c\eta_2\sinh(r_1\eta_2)-\delta \phi(t+r_1)\le 0.\]

\medskip
{\bf Case 3 $0\le t$}: We have 
\begin{align*}
&\phi^{\prime\prime}(t)-c\phi^{\prime}(t+r_1)-\delta \phi(t+r_1)+p \phi(t+(r_1-r_2))e^{-a \phi(t+(r_1-r_2))}\\
&=\frac{-\eta_{2}^{2}}{2}e^{-\eta_2t}-\frac{c\eta_2}{2}e^{-\eta_2 (t+r_1)}+p\left(1-\frac{1}{2}e^{-\eta_{2}(t+(r_1-r_2))}\right)e^{-a \phi(t+(r_1-r_2))}-\delta \phi(t+r_1)\\
&\le \frac{-\eta_{2}^{2}}{2}e^{-\eta_2t}-\frac{c\eta_2}{2}e^{-\eta_2 (t+r_1)}+p\left(1-\frac{1}{2}e^{-\eta_{2}(t+(r_1-r_2))}\right)-\delta \phi(t+r_1)\\
&=\frac{-\eta_{2}^{2}}{2}e^{-\eta_2t}+\frac{c\eta_2}{2}e^{-\eta_2t+\eta_2r_1}-\frac{p}{2}e^{-\eta_2t+\eta_2r_1}+\frac{p}{2}e^{-\eta_2t+\eta_2r_1} -\frac{c\eta_2}{2}e^{-\eta_2t+\eta_2r_1}-\frac{c\eta_2}{2}e^{-\eta_2 (t+r_1)}\\
&+p\left(1-\frac{1}{2}e^{-\eta_{2}(t+(r_1-r_2))}\right)-\delta \phi(t+r_1).
\end{align*} 
\end{proof}
Using the fact that $\eta_2$ is a root of Eq. (\ref{CE1}) we see the following simplification
\begin{align*}
&=\frac{p}{2}e^{-\eta_2t+\eta_2r_1} -\frac{c\eta_2}{2}e^{-\eta_2t+\eta_2r_1}-\frac{c\eta_2}{2}e^{-\eta_2 (t+r_1)}+p\left(1-\frac{1}{2}e^{-\eta_{2}(t+(r_1-r_2))}\right)-\delta \phi(t+r_1)\\
&=-c\eta_2e^{-\eta_2t}\cosh(\eta_2r_1)+\frac{p}{2}e^{-\eta_2t}\left(e^{\eta_2r_1}-e^{\eta_2(r_1-r_2)}\right)-\delta \phi(t+r_1)+p.
\end{align*} 
Noticing that when $r_1\to 0,$ implies $\cosh(\eta_2r_1)\to 1$ and $e^{\eta_2r_1}-e^{\eta_2(r_1-r_2)}\to 1-e^{-\eta_2r_2}$ we can take $r_1, r_2$ small enough such that $e^{\eta_2r_1}-e^{\eta_2(r_1-r_2)} \approx o(r_1-r_2)$. This allows one to see that  for sufficiently small $r_1, r_2$
\begin{align*}
&=-c\eta_2e^{-\eta_2t}\cosh(\eta_2r_1)+\frac{p}{2}e^{-\eta_2t}\left(e^{\eta_2r_1}-e^{\eta_2(r_1-r_2)}\right)-\delta \phi(t+r_1)+p\\
&\approx \left(-c\eta_2+\frac{p}{2}o(r_1-r_2)\right)e^{-\eta_2t}-\delta \phi(t+r_1)+p\\
&\le-c\eta_2+\frac{p}{2}o(r_1-r_2)-\delta \phi(t+r_1)+p.
\end{align*} 
Moreover, \[\lim_{c \to 2\sqrt{p}} -c\mu_1=-p,\]
and \[\lim_{r_1\to 0} \eta_2(r_1) =\mu_1\] we have that
\begin{align*}
&-c\eta_2+\frac{p}{2}o(r_1-r_2)-\delta \phi(t+r_1)+p\\
&\approx \frac{p}{2}o(r_1-r_2)-\delta \phi(t+r_1).
\end{align*}
Thus, we can take $r_1, r_2$ small enough such that \[\frac{p}{2}o(r_1-r_2)<\delta \phi(t+r_1).\] 
This proves the result.

In order to construct quasi lower solutions we will look at a function, defined as \begin{align*}
f(t)&= a(t-T)^3+b(t-T)^2 +\frac{1}{2}.
\end{align*} 
Furthermore, for some large $T>0$ we have the properties:
\begin{enumerate}
\item This bridges smoothly the function $e^{\eta_1 t}/4$ and the constant function $1/2$
\item $f(-T)=(1/4)e^{\eta_1T}$, $f'(-T)= (-\eta_1/4)e^{\eta_1 T}$, $f'(T)=0$, $f(T)=1/2$.
\end{enumerate}
It was shown in Barker and Nguyen, \cite{BarkNguy} that 
\begin{align*}
a&=\frac{\eta_1 Te^{-\eta_1 T}+e^{-\eta_1 T}-2}{16T^3}\\
b&=\frac{-\eta_1 Te^{-\eta_1 T}+6\left( \frac{e^{-\eta_1 T}}{4}-\frac{1}{2}\right)}{8T^2},
\end{align*}
as well as the following claim.
\begin{claim} Define $f(t)$ to be the bridge function from above, then
\begin{equation}\label{bridge2}
\lim_{T\to\infty} \sup_{-T\le t\le T}\max \{ |f'(t)|,|f''(t)|\} =0.
\end{equation}
\end{claim}

\begin{claim}\label{claim 7}
For sufficiently small $r_1$ and $c>2\sqrt{\delta}$, and $\eta_1$ to be the positive root of Eq. (\ref{CE}), then  the function \[\underline\varphi(t)=
\begin{cases} \frac{ e^{\eta_1t}}{4}, \ t< -T ,\\
f(t), \ -T \le  t \le T \\
\frac{1}{2} , \ t> T ,
\end{cases}
 \]  is a quasi-upper solution of (\ref{Wave1}).
\end{claim}
\begin{proof} We will do this in cases just as before. 

\medskip
{\bf Case 1: $t\le -T- r_1$}
\begin{align*}
&\phi^{\prime\prime}(t)-c\phi^{\prime}(t+r_1)-\delta \phi(t+r_1)+p \phi(t+(r_1-r_2))e^{-a \phi(t+(r_1-r_2))}\\
&=\frac{\eta_1^{2}}{2}e^{\eta_1t}-\frac{c\eta_1}{2}e^{\eta_1(t+r_1)}-\frac{\delta}{2} e^{\eta_1(t+r_1)}+\frac{p}{2}e^{\eta_1(t+r_1)}e^{-a \phi(t+(r_1-r_2))}\\
&=\left(\frac{\eta_1^{2}}{2}-\frac{c\eta_1}{2}e^{\eta_1r_1}-\frac{\delta}{2} e^{\eta_1r_1}\right)e^{\eta_1t}+\frac{p}{2}e^{\eta_1(t+r_1)}e^{-a \phi(t+(r_1-r_2))}\\
&=\frac{p}{2}e^{\eta_1(t+r_1)}e^{-a \phi(t+(r_1-r_2))}\ge 0,
\end{align*}
because $\eta_1$ is a root for Eq. (\ref{CE}).

\medskip
{\bf Case 2:  $-T-r_1\le t\le T$}
This case follows from the fact that on this interval
\begin{align*}
 \sup_{-T-r_1 \le t\le T}  |\underline{\varphi}_{1}^{\prime\prime}(t)|\pm c|\underline{\varphi}_{1}^{\prime}(t+r_1)|= \sup_{-T\le t\le T} |f'(t)|\pm c|f''(t)|
\end{align*}
that could be made as small as we like by taking $T$ sufficiently large, and $p>\delta.$ Thus, we have for some large $T$
\begin{align*}
&\phi^{\prime\prime}(t)-c\phi^{\prime}(t+r_1)-\delta \phi(t+r_1)+p \phi(t+(r_1-r_2))e^{-a \phi(t+(r_1-r_2))}\\
&= -\delta \phi(t+r_1)+p \phi(t+(r_1-r_2))e^{-a \phi(t+(r_1-r_2))}+o(T)\ge 0.
\end{align*}
This case has been proven.

\medskip
{\bf Case 3:  $T\le t$}  This case is trivially due to the fact that the function is constant.
\end{proof}
\begin{corollary}
Assume that $c>2\sqrt{p}$ is given. Then, the Eq. (\ref{Wave1}) has a traveling wave solution $u(x,t){\phi(x+ct)}$ for sufficiently small delays $\tau_1,\tau_2$.
\end{corollary}

\section{Numerical Simulations}
In this section we will construct, via a specific example numerical upper and lower solutions for Eq. (\ref{Wave1}). Using the formula found in Theorem (4.1) from Mallet-Peret, \cite{mal} we have
\begin{equation}\label{Green}
G(t,r)= -\frac{1}{2\pi}\int_{-\infty}^{\infty}\frac{e^{i\xi t}}{-\xi^2-ci\xi e^{i\xi r_1}-\delta e^{i\xi r_1}}d\xi
\end{equation}

\medskip
\noindent
A relatively simple numerical scheme should be appropriate due to the nature of the quasi upper and lower solutions and the smoother upper and lower solutions. To this end, we have the following lemma.
\begin{lemma}
Let $G(t,r)$ be the Green's function found in from Eq. (\ref{Green}) , then 
\begin{enumerate}
\item for all $N\in \N, \ t,\xi \in \R$ there is some positive constant, $K$ independent of $N$ such that  $$\left|G(-N,r)+G(N,r)\right|<K.$$
\item  For any small positive number $\varepsilon$ there is some sufficiently large $N \in \N$ dependent upon $\varepsilon$ such that for all $t,\xi \in \R$    
    $$\left|G(-N,r)+G(N,r)\right|<\varepsilon.$$
\end{enumerate}
\end{lemma}
\begin{proof} The proof for both parts are straight forward due to the fact that for all $t\in \R$ there is some positive constant, $\delta_1$ such that $||G(t,r)||\le K_1e^{-\delta_1|t|}.$ In fact, part $i.)$ can be shown via induction on $N$. Take $N=1$, then 
\begin{align*}
\left|G(-N,r)+G(N,r)\right|\le 2K_1e^{-\delta_1}<2K_1.
\end{align*}
Take $K=2K_1$ then the result follows. In order to show part $ii.)$ fix $\varepsilon>0$  then for some $N\in \N$ such that \[N\ge \ln\left({\frac{\varepsilon}{2K_1}}^{\frac{-1}{\delta_1}}\right)\]
we have the following estimate:
\begin{align*}
&\left|G(-N,r)+G(N,r)\right|\le 2K_1 e^{-\delta_1 N}<\varepsilon.
\end{align*}
\end{proof}
\noindent This allows us to disregard the tails of the Green's function and focus on some finite interval in order to numerically approximate the integral for any fixed $t\in \R.$ In fact, the interval can be chosen to be relatively small. We know that the following iteration is convergent \[\phi_n=\int^\infty_{-\infty} G(t-s,r)H(\phi_{n-1}(s))ds,\]
where $H$
We will take $r_1=1, r_2=1/4, p=2, \delta=1, c=2\sqrt{2}, N=50, T=1$.
then we will approximate \[G(t,r)\approx\ -\frac{1}{2\pi}\int_{-50}^{50}\frac{e^{i\xi t}}{-\xi^2-ai\xi e^{i\xi}-be^{i\xi}}d\xi.\]
Furthermore, since $\phi_n\ge 0,$ we have the approximation
\[\phi_n=\left|\int^\infty_{-\infty} G(t-s,r)H(\phi_{n-1}(s))ds\right|\le\int^\infty_{-\infty} \left|G(t-s,r)H(\phi_{n-1}(s))\right|ds.\]
With this in mind, we know that $\left|G(t-s,r)\right|\le K_1e^{-\delta_1|t|}.$ Shifting the line of integration to the parallel line $z=\xi +i|\lambda_1|,$ where $\lambda_1$ is the negative root of Eq. (\ref{CE}) gives the following 
\[\left|G(t-s,r)\right|\le K_1e^{\lambda_1 |t|}.\]

\medskip
A composite Simpson's rule will be used to approximate 
\[K_1\approx \frac{1}{2\pi}\int_{-50}^{50}\frac{1}{\left(\xi +i|\lambda_1|\right)^2+\left[2i\left(\xi +i|\lambda_1|\right)+1\right]e^{(\xi +i|\lambda_1|)}}d\xi.\]
For brevity, we denote 
\[f(\xi)=\frac{1}{2\pi}\int_{-50}^{50}\frac{1}{\left(\xi +i|\lambda_1|\right)^2+\left[2i\left(\xi +i|\lambda_1|\right)+1\right]e^{(\xi +i|\lambda_1|)}},\]
then it is well know that the composite Simpson's rule can be written as \[I_n=\int_{-50}^{50}f(\xi)\approx \frac{h}{3}\left[f(-50)+4\sum_{i=1}^{n/2} f(\xi_{2i-1})+2\sum_{i=1}^{n/2-1} f(\xi_{2i})+f(50),\right]\]
where $h=100/n, \ n$ is the number of sub intervals. The results for various step sizes (calculated in MATLAB) can be found in the table below, as well as plots  for the quasi upper and lower solutions below  

\medskip
\begin{center}
\begin{tabular}{||c c c||} 
 \hline
 $n$ & $h$ & $|I_n|$ \\ 
 \hline \hline
 100 & 1 & .2861  \\ 
 \hline
 1000 & .1 & .3064  \\
 \hline
 10000 & .01 & .3066  \\
 \hline
 100000 & .001 & .3067  \\
 \hline
 1000000 & .0001 & .3067 \\ 
 \hline
\end{tabular}
\end{center}

\medskip

\begin{figure}[h]
\centering
\begin{tabular}{cc}
\includegraphics[width=7cm]{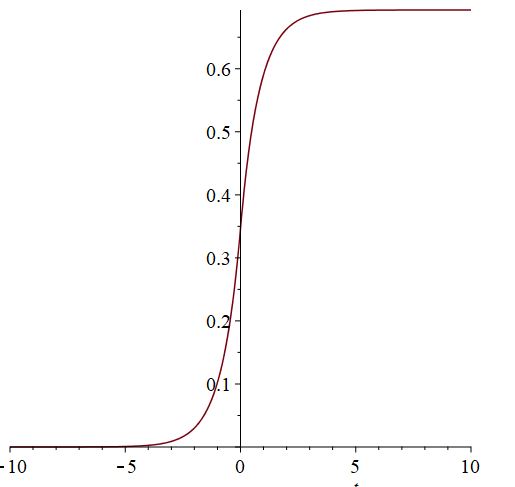}&
\includegraphics[width=6cm]{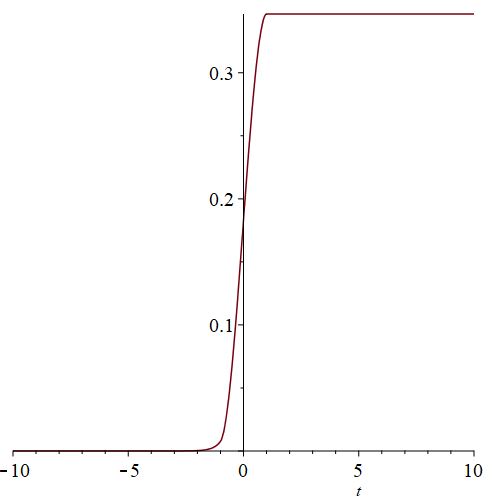}\\
Quasi Upper Solution &\parbox[t]{4cm}{Quasi Lower Solution}
\end{tabular}
\end{figure}

\medskip

\newpage 
\bibliographystyle{amsplain}

\end{document}